\documentclass[12pt]{amsart}
\usepackage{amsfonts,amssymb, amscd, latexsym, graphicx, psfrag, color, float, accents}
\usepackage[all]{xy}

\newtheorem{dummy}{dummy}[section]
\newtheorem{lemma}[dummy]{Lemma}
\newtheorem{theorem}[dummy]{Theorem}

\newtheorem{corollary}[dummy]{Corollary}
\newtheorem{proposition}[dummy]{Proposition}
\theoremstyle{definition}
\newtheorem{definition}[dummy]{Definition}

\newtheorem{remark}[dummy]{Remark}

\newcommand{\bQ}{\mathbb{Q}}

\newcommand{\cA}{\mathcal{A}}
\newcommand{\cB}{\mathcal{B}}
\newcommand{\cC}{\mathcal{C}}
\newcommand{\cD}{\mathcal{D}}
\newcommand{\cE}{\mathcal{E}}
\newcommand{\cF}{\mathcal{F}}
\newcommand{\cG}{\mathcal{G}}
\newcommand{\cH}{\mathcal{H}}
\newcommand{\cI}{\mathcal{I}}

\newcommand{\cO}{\mathcal{O}}
\newcommand{\cP}{\mathcal{P}}
\newcommand{\cQ}{\mathcal{Q}}
\newcommand{\cR}{\mathcal{R}}

\newcommand{\cS}{\mathcal{S}}
\newcommand{\cT}{\mathcal{T}}

\newcommand{\cX}{\mathcal{X}}
\newcommand{\cY}{\mathcal{Y}}
\newcommand{\cZ}{\mathcal{Z}}

\newcommand{\Perf}{\mathcal{P}\mathrm{erf}}

\begin{document}

\title[Derived loop stacks and categorification of orbifold products]
{Derived loop stacks and categorification of orbifold products}

\begin{abstract}
The existence of interesting multiplicative cohomology theories for orbifolds was first suggested by string theorists, and orbifold products have been intensely studied by mathematicians for the last fifteen years. %, and several distinct but related constructions are now available: 
In this paper we focus on the \emph{virtual orbifold product} that was first introduced in Lupercio et al. (2007). We construct a categorification of the virtual orbifold product that leverages the geometry of derived loop stacks. %as studied by Ben-Zvi Francis and Nadler. 
%Among its applications, 
By work of Ben-Zvi Francis Nadler, this reveals connections between virtual orbifold products and  Drinfeld centers of  monoidal categories, thus   answering a question of Hinich. 
  \end{abstract}

\author{Sarah Scherotzke}
\address{Sarah Scherotzke, 
Mathematical Institute of the University of Bonn, 
Endenicher Allee 60, 
53115 Bonn,
Germany}
\email{sarah@math.uni-bonn.de}

\author{Nicol\`o Sibilla}
\address{Nicol\`o Sibilla, Department of Mathematics,    University of British Columbia, 1984 Mathematics Road, 
Vancouver, B.C.
Canada V6T 1Z2
}
\email{sibilla@math.ubc.ca}

\maketitle

{\small \tableofcontents}

\section{Introduction}

The existence of non-trivial  multiplicative cohomology theories for orbifolds was first suggested by  work of string theorists   
\cite{Z}. Chen and Ruan \cite{CR} gave a mathematical formalization of these ideas: 
the Chen-Ruan cohomology of an orbifold $\cX$, $H_{CR}^*(\cX)$, is a graded $\bQ$-algebra that is linearly isomorphic to the cohomology 
of the inertia orbifold $I \cX$,  but carries a non-trivial associative product (the \emph{orbifold product}) defined in terms of
 the degree zero Gromov-Witten theory of $\cX$. 
After work of several authors this theory was recast in the language of 
algebraic geometry  \cite{FG, AGV1, AGV2}, and the definition of orbifold products was extended to many different cohomology theories. In  \cite{JKK} Jarvis, Kaufmann and Kimura defined the  orbifold K-theory of a global quotient DM stack 
$\cX$: as in the case of Chen-Ruan cohomology, orbifold K-theory is isomorphic as a linear space to the algebraic K-theory of the inertia orbifold, $I \cX$, but is equipped with a non-trivial orbifold product.

It was later realized that orbifold cohomology theories admit in fact a rich web of distinct multiplicative structures (called   \emph{inertial products} in \cite{EJK2}) 
that are governed by various virtual bundles on the double inertia stack $I^2 \cX$: the orbifold product is only one of them. An especially important variant of the orbifold product is the \emph{virtual orbifold product} introduced in \cite{LUX, LUX+}, and investigated in \cite{EJK1, EJK2} from the perspective of algebraic geometry. We denote $*^{virt}$ the virtual orbifold product, and 
$K^{virt}(\cX):= K_0(I \cX, *^{virt})$ the virtual orbifold K-theory of $\cX$. In this paper we achieve a categorification of virtual orbifold K-theory. 

Let 
$\cX$ be a smooth DM stack, and assume that $\cX$ admits a presentation as a global quotient $\cX = [X/G]$, where $X$ is an affine scheme and $G$ is a linear group. Denote $L \cX$ the derived loop stack of $\cX$ in the sense of \cite{TV, BFN}. The bounded derived category of $L \cX$, 
$D^b(Coh(L \cX))$ carries a braided monoidal structure  which was defined in  \cite{BFN}: we denote it $\otimes^{str}$, and denote $G_0(L \cX, \otimes^{str})$ 
the Grothendieck group of $D^b(Coh(L \cX))$ together with the commutative product induced by $\otimes^{str}$. 
The following is our main theorem.

\begin{theorem}
\label{main:loop}
Let $\iota: I \cX \rightarrow L \cX$ be the natural inclusion. Then 
$\iota_*$ gives rise to an isomorphism of rings:  
$$
\iota_*: K^{virt}(\cX) = K_0(I \cX, *^{virt}) \stackrel{\cong}{\rightarrow} G_0(L \cX, \otimes^{str}).
$$
\end{theorem}

Using the the results of \cite{BFN} Theorem \label{main:loop} can be reformulated as follows.

\begin{theorem}
\label{main:drinfeld}
Let 
$\cT r(D^b(Coh(\cX)))$ be the derived Drinfeld center (in the sense of \cite{BFN}) of the the symmetric monoidal category $D^b(Coh(\cX))$. Then there is a natural   isomorphism:
$$
K^{virt}(\cX) \cong K_0(\cT r(D^b(Coh(\cX))).
$$
\end{theorem}

This result has several useful  consequences, we list some below:
\begin{itemize}
\item As $K^{virt}(\cX)$ can be realized as the Grothendieck group of a braided monoidal category, it  carries a natural structure of 
$\lambda$-ring. This recovers results of \cite{EJK2}.
\item The prescription in \cite{EJK2} gives a definition of virtual orbifold cohomology for global quotient DM stacks. By setting $K^{virt}(\cX):=G_0(L \cX, \otimes^{str})$ we obtain a definition of virtual orbifold cohomology that applies to all  DM stacks with finite stabilizers (and in fact, to a very large class of  derived $\infty$-stacks).

\item Since $*^{virt}$ lifts to a tensor  product on $D^b(Coh(L \cX))$ it induces a  multiplicative structure on the full G-theory spectrum of $L \cX$, which is equivalent to the K-theory spectrum of $I \cX$, $G_*(L \cX) \cong K_*(I \cX)$.\footnote{Recall that if $\cX$ is a derived stack, we refer to the spectrum $K_*(\Perf(\cX))$ as the (algebraic) K-theory of $\cX$, and to $K_*(D^b(Coh(\cX)))$ as the G-theory of $\cX$. The equivalence $G_*(L \cX) \cong K_*(I \cX)$ comes from Barwick's Theorem of the Heart \cite{Ba}.} This is a much richer invariant than the virtual orbifold K-theory, which can be recovered by taking $\pi_0$, $K^{virt}(\cX) = \pi_0(K_*(I \cX))$. 
Also in this way we achieve a fully motivic definition of the virtual product, which is therefore not confined to K-theory but extends to any linear invariant of stable categories: for instance, our result enables the definition of virtual orbifold products on Hochschild homology 
and negative cyclic homology.  
\end{itemize}

Our initial motivations came 
from a proposal of Hinich. In \cite{Hi} Hinich proves  the that abelian category of coherent sheaves over $I \cX$  is isomorphic to the (\emph{underived})  Drinfeld center \cite{JS} of the abelian tensor category of coherent sheaves over $\cX$, and notes that $Coh(I \cX)$ inherits from this equivalence an interesting braided tensor product. He then asks  whether this would give an alternative description of the orbifold product of \cite{JKK} on $K_0(I \cX)$. Theorem \ref{main:drinfeld} implies that the answer to Hinich's question is negative: the tensor product of the Drinfeld center of $Coh(\cX)$ does not descend to the orbifold product on  $K_0(I \cX)$, but rather to the \emph{virtual orbifold product}. 
The two are almost always different: notable exceptions include the case of classifying stacks of finite groups and of smooth schemes. %($[*/G]$, where $G$ is a finite group, 
%and the case when $\cX = X$ is a scheme. 
For classifying stacks of finite groups %We remark that for 
%$[*/G]$ 
a different but related connection between orbifold products and Drinfeld doubles was studied by Kaufmann and Pham \cite{KP}. 
Our work was also inspired by ideas of Manin and To\"en on  categorification of quantum cohomology, see \cite{Ma} and \cite{To1} Section 4.4 (6). Some recent work in this direction can be found in the preprint of To\"en \cite{To2}. \\ \\

{\bf Acknowledgments: } 
We thank Ralph Kaufmann and Timo Sch\"urg for inspiring conversations in the early stages of this  project. We are grateful to Kai Behrend and David Carchedi for useful  discussions and for answering our many questions.

\section{Preliminaries} 

\subsection{$\infty$-categories}
It is well known that triangulated categories are not well adapted to capture many important 
functoriality properties of categories of sheaves. Fortunately various possible ways to obviate these deficiencies are now available:  in the early $90$-s,  Bondal and Kapranov \cite{BoK}  
proposed the formalism of quasi-triangulated dg categories as a better behaved replacement of ordinary triangulated category theory for the purposes of algebraic geometry. In this paper we work with a different enhancement of triangulated categories, that is provided by  stable $\infty$-categories. 
Our model of choice of $\infty$-categories are quasi-categories, which were developed by Joyal \cite{Jo} and have been extensively investigated by Lurie \cite{Lu}. For brevity we always refer to quasi-categories simply as  $\infty$-categories, 
the reader should consult \cite{Lu} 
for basic notations and definitions.   
We remark that in this paper we are exclusively interested in characteristic $0$ applications: under this assumption the theory of stable $\infty$-categories is equivalent to the theory of triangulated dg categories \cite{Co}.  

\subsection{Derived algebraic geometry}
\label{sec:dag}
 In the following it will be often important to consider spaces of maps from simplicial sets to algebraic geometric objects such as schemes and DM stacks. 
Derived algebraic geometry provides a language in which to make sense of these constructions. 
We work over a ground field $\kappa$ of characteristic $0$.
A careful 
definition 
of derived stacks can be found in \cite{To1}. For an agile exposition of this material see \cite{BFN}  
Section 2.3, which employs  as we do the language of $\infty$-categories. Let  $d\cA lg_\kappa$ be the $\infty$-category of simplicial commutative 
$\kappa$-algebras. The opposite category of $d \cA lg_\kappa$, which we denote $d \cA ff_\kappa$, is a site with the \'etale topology (see Section 2.3 of \cite{BFN}): we denote it $(d \cA ff_\kappa)_{\acute{e}t}$. 
Derived stacks are sheaves over 
$(d \cA ff_\kappa)_{\acute{e}t}$ 
with 
values in the $\infty$-category of topological spaces, $\cT op$.

Derived stacks form the $\infty$-category $d\cS t_\kappa$.  Some important examples of derived stacks are: 

\begin{itemize}
\item ordinary schemes and stacks of groupoids (in the following, we will refer to these simply as schemes and stacks),

\item topological spaces (that are viewed as constant sheaves of spaces), and more generally  underived \emph{higher} stacks   \cite{To1},

\item derived affine schemes, that is objects of $d \cA ff_\kappa$.
\end{itemize} 
 There exists a truncation functor $t_0(-)$ that maps derived stacks to underived  stacks: if $F$ is a derived stack, there is a  canonical closed embedding $t_0(F) \rightarrow F$. 
All limits and colimits of derived stacks are taken in the $\infty$-category $d\cS t_\kappa$, that is, they are always \emph{derived}. We point out that this also applies to limits and colimits of schemes: for instance, if $X \rightarrow Y \leftarrow Z$ is a diagram of schemes, $X \times_Y Z$ denotes the \emph{derived} fiber product of $X$ and $Z$, which in general differs  from the ordinary fiber product. The ordinary fiber product can be recovered as  
$t_0(X \times_Y Z)$.

It is often useful to probe the  geometry of derived stacks by mapping spaces into them. An especially important construction of this kind is the derived loop stack. As in ordinary topology,  if $\cX$ is in $d \cS t_\kappa$, we define the loop stack of $\cX$ to be the space of maps from $S^1$ into $\cX$, $\cX^{S^1}$. We denote the loop stack $L \cX$.  
Recall that $S^1$ can be realized as  the colimit of the diagram  
$* \leftarrow (* \coprod *) \rightarrow *$ in $\cT op$: 
this captures the fact that a circle can be obtained by joining two intervals at their endpoints. 
As a consequence, $L \cX$ is equivalent to the fiber product of the diagonal $\cX \stackrel{\Delta}{\rightarrow} \cX \times \cX$ with itself, 
$
L\cX \cong \cX \times_{\cX \times \cX} \cX.$ Note that even in the case of ordinary schemes, loop stacks have non-trivial derived structure: in fact if $\cX$ is a smooth scheme $L \cX$ is equivalent to the total space of the \emph{shifted} tangent bundle $T\cX[-1]$, this is a form of the Hochschild-Konstant-Rosenberg isomorphism. 
 \begin{proposition}
Let $\cX$ be a DM stack, then $t_0(L \cX)$ is isomorphic to the inertia stack $I \cX$. Further, if $\cX$ is the global quotient of a smooth scheme by a finite group, there is an equivalence $L \cX \cong T I \cX[-1]$  (\cite{ACH} Theorem 4.9).
 \end{proposition}

We can attach to  derived stacks various categories of sheaves. Quasi-coherent sheaves on a derived stack $\cX$ form a presentable and stable symmetric  monoidal $\infty$-category, that we denote $\cQ Coh(\cX)$. The formalism of six operations: $f^*, f_*, f^!, f_!, \otimes, \cH om(-,-)$ carries over to this setting: we note that, contrary to what happens in ordinary triangulated category theory, functors are always \emph{derived}. The category of perfect complexes on $\cX$, $\Perf(\cX)$, is the subcategory of  compact objects in 
$\cQ Coh(\cX)$. 
If $\cX$ satisfies some additional assumptions (e.g. if it is a derived DM stack) $\cQ Coh(\cX)$ can be equipped with a canonical t-structure, we denote its heart $qcoh(\cX)$. Leveraging the existence of the canonical t-structure it is possible to define coherent sheaves: they are quasi-perfect and 
quasi-truncated objects in $\cQ Coh(\cX)$, see \cite{Lu8} Definition 2.6.20. Coherent sheaves form a full stable subcategory $\cC oh(\cX)$ of 
 $\cQ Coh(\cX)$. The canonical t-structure on $\cQ Coh(\cX)$ restricts to a bounded t-structure on $\mathcal{C}oh(\cX)$ with heart 
$coh(\cX)$.
\begin{proposition}[\cite{Lu} Remark 2.3.20]
\label{prop:heart}
Let $\cX$ be a derived DM stack, 
 let $t_0(\cX)$ be its underlying ordinary DM stack and let $\iota: t_0(\cX) \rightarrow \cX$ be the natural embedding. Then,
\begin{itemize}
\item 
$
\iota_*: qcoh(t_0\cX) \rightarrow qcoh(\cX)
$ is an equivalence.
\item 
$
\iota_*: coh(t_0\cX) \rightarrow coh(\cX)
$ is an equivalence.
\end{itemize}
\end{proposition}

The \emph{$K$-theory} of a derived DM stack $\cX$ is the K-theory of its category of perfect complexes, while its \emph{$G$-theory} is by definition the K-theory of $\cC oh(\cX)$,  $G_*(\cX) := K_*(\cC oh(\cX))$. We refer the reader to \cite{Ba2, BGT} for foundations on the K-theory of $\infty$-categories. Recall also that if $\cX$ is a smooth ordinary DM stack there is an equivalence $\cC oh(\cX) \cong \Perf(\cX)$, and therefore the $G$-theory and $K$-theory of $\cX$ are naturally identified.

Combined with Barwick's  ``theorem of the heart'' (which, in the setting of triangulated categories, was originally due to Neeman ), Proposition \ref{prop:heart} has the important corollary that the $G$-theory of $\cX$ and of  $t_0(\cX)$ are equivalent, see \cite{Ba} Proposition $9.2$.  

\begin{corollary}
\label{cor:heart}
There is an equivalence of spectra 
$$
\iota_*: G_*(t_0(\cX)) = K_*(\cC oh(t_0(\cX))) \stackrel{\cong}{\rightarrow} 
G_*(\cX) = K_*(\cC oh(\cX)).
$$
In particular $\iota_*: G_0(t_0(\cX)) \rightarrow G_0(\cX)$ is an isomorphism of groups that  sends the class $\sum\limits_{i=0}^\infty (-1)^ i  \pi_i(\cO_\cX) \in  G_0(t_0(\cX))$ to the class of $\cO_\cX$ in $G_0(\cX)$. 
\end{corollary}

\subsection{Derived Drinfeld center and convolution tensor product}
\label{sec:ddc}
Here we review some results from  \cite{BFN} that will play a key role in the following. 
Denote $\cP r^L$ be the closed symmetric monoidal $\infty$-category of presentable $\infty$-categories (and left adjoint functors between them). 
Let $\cX$ be a \emph{perfect} derived stack in the sense of \cite{BFN} Definition 3.2. For instance, smooth DM stacks are perfect. 
As in ordinary algebra, it is possible to make sense of the Hochschild homology and cohomology of $\cQ Coh(\cX)$ as an associative algebra object in $\cP r^L$. Following \cite{BFN}, we call these respectively the derived trace and derived center of $\cQ Coh(\cX)$, and denote them $\cT r(\cQ Coh(\cX))$ and $\cZ (\cQ Coh(\cX))$. The derived center $\cZ (\cQ Coh(\cX))$ is an \emph{$\cE_2$-category} with the convolution tensor product  
$-\otimes^{conv}-$: recall that 
$\cE_2$-categories are the analogue in $\infty$-category theory of braided monoidal categories.

Let $P$ be the two-dimensional pair of pants, that is, $P$ is a genus $0$ compact surface with three boundary components. Set $P \cX := \cX^P$, and note that restriction to the boundary components gives  maps,
$$
\xymatrix{
& P \cX \ar[dl]_{p_1} \ar[d]^{p_3} \ar[dr]^{p_2} & \\
L\cX & L\cX & L\cX. 
}
$$
We extract from this diagram a non trivial tensor product on $\cQ Coh(L \cX)$, that we denote $\otimes^{str}$:   
$\cF \otimes^{str} \cG = p_{3*}(p_1^*(\cF) \otimes p_2^*(\cG))$.\footnote{The notation $\otimes^{str}$ is motivated by the connection with string topology \cite{CS}, \cite{CJ}, for a discussion of these aspects  see Section 6.1 of \cite{BFN}.} 
\begin{theorem}[\cite{BFN} Proposition 5.2, 6.3, 6.6] 
$(\cQ Coh(L \cX), \otimes^{str})$ is an $\cE_2$-category, and there is an  equivalence of $\cE_2$-categories: 
$$
(\cZ(\cQ Coh(\cX)), \otimes^{conv}) 
 \cong ( \cQ Coh(L \cX), \otimes^{str}).
$$
\end{theorem}

Now assume that $\cX$ is a smooth DM stack.  Note that $Coh(\cX)$ is an associative algebra object in the closed symmetric monoidal $\infty$-category of small, stable and split closed $\infty$-categories. Also, the  tensor product $\otimes^{str}$ restricts to a $\cE_2$-structure on 
$Coh(L \cX)$.

\begin{corollary}
There is an equivalence of $\cE_2$-categories:  
$$
(\cZ(Coh(\cX)), \otimes^{conv}) 
 \cong (Coh(L \cX), \otimes^{str}).
$$
\end{corollary}

\section{A proof of the main theorem for classifying  stacks and schemes}

In this section we work out the two simplest examples of our main theorem: we prove it for DM stacks of the form $BG = [*/G]$, where $G$ is a  finite group, and for smooth schemes.  
In the case of $BG$, a direct proof 
follows from results scattered in the literature: we thought it might be useful to sketch it here. 
Although the proof for schemes does not differ in any essential way from the general argument, it has the advantage that it can be entirely carried out leveraging simple geometric properties of mapping spaces. A similar geometric  treatment of the general case is also possible but is more intricate, and we will not pursue it: however 
those geometric ideas motivate the complete proof of Theorem \ref{main:loop} that we will give in Section \ref{sec:proof}, 
 and might contribute to clarify it.
  
\subsection{Classifiying stacks of finite groups}

Let $G$ be a finite group and let 
$\cX = [*/G]$ be the classifying stack of $G$. 

\begin{remark}
Note that for classifying stacks of finite groups $K^{virt}(\cX)$ 
is equal to $K^{orb}(\cX)$ the \emph{orbifold K-theory} of 
$\cX$ defined in \cite{JKK}. 
This is an immediate consequence of the definitions,  see \cite{EJK2} Section 4.3.  
\end{remark}

\begin{proposition}
There is an isomorphism 
$
K^{virt}(\cX) \cong G_0(L \cX, \otimes^{str}).
$
\end{proposition}
\begin{proof}
We sketch a proof of this fact that draws from different results available in the literature, we leave some details to the reader. Note that since $\cX$ is isomorphic to  $[*/G]$ and $G$ is finite, the diagonal map $\Delta: \cX \rightarrow \cX \times \cX$ is flat. As a consequence there in an equivalence $\cX \times_{\cX \times \cX} \cX \cong t_0(\cX \times_{\cX \times \cX} \cX)$: that is, the loop stack $L \cX$ is equivalent to the inertia stack $I \cX$. Further the convolution tensor product $\otimes^{conv}$ on 
$\cZ (Coh(\cX)) \cong Coh(I \cX)$, restricts to an exact tensor product on $coh(I \cX)$. 

Denote $Rep(G)$ be the abelian monoidal category of representations of $G$, and note that there is a natural isomorphism $Rep(G) \cong coh(\cX)$. The tensor product $\otimes^{conv}$ on $coh(I \cX)$ coincides with the one considered  in \cite{CJ, Hi}: it is the braided tensor product on 
$coh(I \cX)$ that comes from 
the identification between $coh(I \cX)$ and the (\emph{underived}, as defined in \cite{CJ}) Drinfeld center of the monoidal abelian category $coh(\cX) \cong Rep(G)$. Let us unpack this observation a bit more.

If $D(\kappa[G])$ is the Drinfeld double of the group algebra of G, its abelian category of representations $D(\kappa[G])-mod$ is equipped with a braided monoidal structure. The discussion in the previous paragraph shows that there is an equivalence of braided tensor categories:  
$$
(Coh(I \cX), \otimes^{conv}) \cong D(\kappa[G])-mod.  
$$
In particular, there is an isomorphism of rings 
$K_0(L\cX, \otimes^{str}) \cong K_0(D(\kappa[G])-mod)$.  Thus the claim would follow if could prove that  $K_0(D(\kappa[G])-mod)$ is isomorphic to  
$K^{virt}(\cX) \cong K^{orb}(\cX)$: this has been established by Kaufmann and Pham, see \cite{KP} Theorem 3.13.  
\end{proof}

\subsection{Smooth schemes}

Let $X$ be a smooth scheme.  
Then $I X = X$ and all the inertial products in $K_0(X)$ defined in \cite{EJK2} coincide with the ordinary product in the  K-theory of $X$. 
Denote $\otimes$ the ordinary symmetric tensor product on $\cQ Coh(X)$. In this section we prove that  $\iota_* : K^{virt}(X) = K_0(X, \otimes) \rightarrow  G_0(LX, \otimes^{str})$ is an isomorphism of rings. Note that by Proposition \ref{prop:heart} the map $\iota_*$ is a group isomorphism.  
The following proposition shows that $\iota_*$ is also compatible 
with the product structures. This implies Theorem \ref{main:loop} for smooth schemes. 
\begin{proposition}
\label{prop:schemes}
Let $\cF$ and $\cG$ be in $\cQ Coh(X)$. 
Then there is a natural equivalence: 
$$
(\iota_*\cF) \otimes^{str} (\iota_*\cG) \cong \iota_*(\cF \otimes \cG).
$$
\end{proposition}

\begin{corollary}
There is an isomorphism $\iota_*: K^{virt}(X) = K_0(X, \otimes) \rightarrow  G_0(LX, \otimes^{str})$. 
\end{corollary}
 
Before proceeding with the proof of  Proposition \ref{prop:schemes}, we make some preliminary observations that take place in 
$\cT op$. Denote $D$ the closed disc, and let $P$ be the pair of pants. It is useful to model $P$ as the complement of three non-intersecting open discs, $D_1$ $D_2$ and $D_3$, in the 2-sphere $S^2$. Denote $b_1, b_2,  b_3 : S^1 \rightarrow P$ the inclusions given by the identification $S^1 = \partial D_i$. Let $\{i, j, k \} = \{1, 2, 3\}$, and denote $P_i = S^2 - (D_j \cup D_k)$, and $P_{i, j} = S^2 - D_k$. Note that in $\cT op$ we have equivalences $P_i \cong S^1$, $P_{i,j} \cong D$.

\begin{lemma}
\label{lem:1}
The following diagrams of inclusions,

\xymatrix{
&&&&& P \ar[d] \ar[r] & P_i \ar[d] && S^1 \ar[d]_{b_i} \ar[r] & D_i \ar[d]\\ 
&&&&& P_j \ar[r] & P_{i,j}, && P \ar[r] & P_i,}

are push-outs in $\cT op$.
\end{lemma}

\begin{lemma}
\label{lem:3}
Let $X \stackrel{i}{\rightarrow} Z \stackrel{j}{\leftarrow} Y$ be maps of quasi-compact and quasi-separated derived DM stacks. Denote  $l_X : X \times_Z Y \rightarrow X$ and $l_Y : X \times_Z Y \rightarrow Y$ the projections, and set   
$l_Z = i \circ l_X \cong j \circ l_Y$. Let $\cF$ be in $\cQ Coh(X)$ and let 
$\cG$ be in $\cQ Coh(Y)$. Then  there is a natural equivalence  
$i_*\cF \otimes j_*G = l_{Z*}(l_X^*\cF \otimes l_Y^* \cG)$   in $\cQ Coh(Z)$.
\end{lemma}
\begin{proof}
There is a chain of natural equivalences: 
$$
i_*\cF \otimes j_*\cG \cong i_*(\cF \otimes i^*j_*\cG) \cong i_*(\cF \otimes l_{X*}l_Y^*\cG) \cong i_*l_{X*}(l_X^*\cF \otimes l_Y^*\cG) = l_{Z*}(l_X^*\cF \otimes l_Y^*\cG).
$$
The first and third equivalences follow from the projection formula, and the second follows from the base change formula of  \cite{To4} Proposition 1.4. 
\end{proof}

\begin{proof}[Proof of Proposition \ref{prop:schemes}]
Denote $\iota: X = t_0(L X) \rightarrow LX$ the natural  embedding. Note that $\iota$ can be described as the restriction map $X \cong X^D \rightarrow X^{S^1}$. Consider the diagram  
$$
\xymatrix{
& & X^{P_{12}} \ar[rd]^{s_2} \ar[ld]_{s_1} & & \\
& X^{P_1} \ar[dl]_{n_1} \ar[r]^{u_1} & X^P \ar[dl]^{p_1} \ar[d]^{p_3} \ar[dr]_{p_2} & X^{P_2} \ar[l]_{u_2} \ar[dr]^ {n_2} & \\
X^D \cong X \ar[r]^{i} & X^{S^1} & X^{S^1} & X^{S^1} & X^D \cong X \ar[l] .
}
$$
By Lemma \ref{lem:1} the top triangle and the right and left squares are all fiber products. The base change formula (see \cite{To4} Proposition 1.4)  implies that we have equivalences $p_1^*\iota_* \cF \cong u_{1*}n_1^*\cF$ and $p_2^*\iota_*\cG \cong u_{2*}n_2^*\cF$. 
Using Lemma \ref{lem:3} we 
can write 
$$
p_1^*\iota_* \cF \otimes p_2^*\iota_* \cG \cong u_{1*}n_1^*\cF \otimes u_{2*}n_2^*\cG \cong u_{1*}s_{1*}(s_1^{*}n_1^*\cF \otimes s_2^{*}n_2^*\cG) \cong u_{1*}s_{1*}(\cF \otimes \cG),
$$
where the last equivalence follows from the fact that, since 
$P_{12} \cong D \cong *$, 
$X^{P_{12}} \cong X$ and $(n_i \circ s_i)^* \cong Id$. 
Thus, 
$
\iota_*\cF \otimes^{str} \iota_*\cG =  p_{3*}(p_1^* \iota_*\cF \otimes p_2^*\iota_*\cG) \cong p_{3*}u_{1*}s_{1*}(\cF \otimes \cG) \cong \iota_*(\cF \otimes \cG),
$
and this concludes the proof.
\end{proof}

\section{Some observations on mapping stacks}

Let $\cX$ be a derived stack. In this section we collect some facts about the mapping stacks $L \cX$ and 
$P \cX$. There exists an evaluation map $S^1 \times L\cX \rightarrow L \cX$. Fixing a point on $S^1$ we get a map $L \cX \rightarrow \cX$ that we denote $ev$.
\begin{lemma}
\label{lem:fiber}
The following diagrams are fiber products in $d \cS t_\kappa$:
 $$ 
\xymatrix{
  P \cX \ar[d] \ar[r] & \cX \ar[d]^\Delta &&  P \cX \ar[d] \ar[r] & L \cX \ar[d]^{ev} \\
   \cX \ar[r]^{\Delta} & \cX \times \cX \times \cX, &&  L \cX \ar[r]^{ev} & \cX.}
$$
\end{lemma}
\begin{proof}
Note that $P$ is equivalent to a wedge of two circles in $\cT op$. Thus $P$ can be expressed as the push-out of the following two diagrams: $
* \leftarrow * \amalg * \amalg * \rightarrow *, 
$
and  $S^1 \leftarrow * \rightarrow S^1$, the claim follows from here. \end{proof}

\begin{corollary}
Let $\cX$ be a DM stack, and denote $I^2 \cX$ the \emph{double inertia stack} of $\cX$. Then $t_0(P \cX) \cong I^2 \cX$. 
\end{corollary}
\begin{proof}
Recall that $I^2 \cX$ is the \emph{ordinary} fiber product of the diagram 
$I \cX \rightarrow \cX \leftarrow I \cX$. 
The corollary is a consequence of the general fact that if $\cX \rightarrow \cZ \leftarrow \cY$ is a diagram 
in $d \cS t_\kappa$, then 
$$
t_0(\cX \times_\cZ \cY) \cong t_0(t_0(\cX) \times_{t_0(\cZ)} t_0(\cY)).
$$ 
 Note also that if $t_0(\cX), t_0(\cY), t_0(\cZ)$ are schemes or stacks, $t_0(t_0(\cX) \times_{t_0(\cZ)} t_0(\cY))$ coincides with their ordinary fiber product. Thus we have equivalences,   
$
t_0(P \cX) \cong t_0(L \cX \times_\cX L \cX) \cong t_0 (t_0(L \cX) \times_\cX t_0(L \cX)) = t_0(I \cX \times_\cX I \cX) \cong I^2 \cX.
$
\end{proof}

\begin{remark}
\label{rem:group}
Denote $q_1, q_2: I^2 \cX \rightarrow \cX$ the two projection maps. Further, in the category of \emph{ordinary} stacks, $I^2 \cX$ is a 
$\cX$-group. A discussion of this can be found in Remark 79.5.2 \cite{SP}. Multiplication and inverse are encoded in the two maps $\mu: I^2 \cX \rightarrow I \cX$ and $(-)^{-1} : I \cX \rightarrow I \cX$. We set $q_3 := \mu$. By Lemma \ref{lem:fiber} the derived stack $P \cX$ carries two projections $p_1, p_2: P \cX \rightarrow L \cX$. Note that these coincide with the restriction to two of the boundary components of $P$, and  the restriction to the third boundary gives a morphism $p_3: P \cX \rightarrow L \cX$ (see Section \ref{sec:ddc}).  
All these maps fit in a commutative diagram: 
$$
\xymatrix{
& I^2 \cX \ar[d] \ar[ld]_{q_1} \ar[rd]^{q_3} \ar[rrrd]^{q_2} & & &  \\
I \cX \ar[d] & P \cX \ar[ld]_{p_1} \ar[rd]^{p_3} \ar[rrrd]^{p_2} & I \cX \ar[d] & & I \cX \ar[d]
\\ 
L \cX &    & L \cX & & L \cX,  
}
$$
where the vertical arrows are given by the natural embedding $I \cX = t_0(L \cX) \rightarrow L\cX$ and $I^2 \cX = t_0(P \cX) \rightarrow P \cX$. \end{remark}

Let $G$ be an algebraic group acting on a scheme $X$, and let $\cX = [X/G]$ be the quotient stack. We let $X \times G \rightarrow X \times X$ and $X \times G \times G \rightarrow X \times X \times X$ be the maps defined on closed points by the assignment $(x, g) \mapsto (x, gx)$ and $(x,g,h) \mapsto (x, gx, hx)$.  The next Lemma gives an explicit construction of $L \cX$ and $P \cX$ as global quotients of \emph{derived} schemes. 

\begin{lemma}
\label{lem:quot}
\begin{itemize}
\item Let $L_GX$ be the derived scheme obtained as the following fiber product:
$$
\xymatrix{
L_GX \ar[d] \ar[r] & X \ar[d]^{\Delta} \\
X \times G \ar[r] & X \times X.}
$$
Then there is a natural action of $G$ on $L_GX$ and $L\cX$ is isomorphic to 
$[L_GX/G]$.
\item Let $P_G X$ be the derived scheme obtained as the following fiber product:
$$
\xymatrix{
P_GX \ar[d] \ar[r] & X \ar[d]^{\Delta} \\
X \times G \times G \ar[r] & X \times X \times X.}
$$
Then there is a natural action of $G$ on 
$P_GX$ and $P \cX$ is isomorphic to $[P_GX/G ]$.
\end{itemize}
\end{lemma} 
\begin{proof}
The first part of the Lemma is stated without proof in Section 4.4 of \cite{To3}. We include a proof for completeness.
Consider the diagram, 
$$
\xymatrix{
L_GX \ar[d] \ar[r] & G \times X \ar[d] \ar[r] & \cX  \ar[d] \\
X  \ar[r] & X \times X \ar[r] & \cX \times \cX.}
$$
There is an equivalence $G \times X  \cong X \times_\cX X$. Standard properties of fiber products imply that there is an equivalence 
$G \times X \cong (X \times X) \times_{\cX \times \cX} \cX$, and therefore that the right square is a fiber product. The left square is a fiber product by the definition of $L_GX$. Thus the exterior square is a  fiber product as well.

Next note that the left square in the diagram 
$$
\xymatrix{
L_GX \ar[d] \ar[r] & L\cX \ar[d] \ar[r] & \cX  \ar[d] \\
X  \ar[r] & \cX \ar[r] & \cX \times \cX,}
$$
is a  fiber product. Indeed, both the right and the exterior squares are fiber products: the right square is a  fiber product by the discussion  in Section \ref{sec:dag}, and the fact that the exterior square is also a fiber product was proved in the previous paragraph.  
This and the fact that $\cX$ is equivalent to $[X/G]$ prove that both the right and the left squares in 
$$
\xymatrix{
L_GX \ar[d] \ar[r] & X \ar[d] \ar[r] & {*} \ar[d] \\
L\cX  \ar[r] & \cX \ar[r] & [*/G],
}
$$
are fiber products, and that therefore the exterior square is as well. Thus $L \cX$ is equivalent to $[L_GX/G]$, as we needed to show.

The second part of the Lemma is proved in a very similar way. Note that the fiber product of the diagram 
$
X \times X \times X \rightarrow \cX \times \cX \times \cX \leftarrow \cX $ is equivalent to $X \times_\cX X \times_\cX X \cong X \times G \times G$. Next consider the diagram, 
$$
\xymatrix{
P_GX \ar[d] \ar[r] & G \times G \times X \ar[d] \ar[r] & \cX  \ar[d] \\
X  \ar[r] & X \times X \times X \ar[r] & \cX \times \cX \times \cX.}
$$
The exterior square is a fiber product, as both  right and left squares are. As in the proof of the first part of the Lemma, leveraging this and Lemma \ref{lem:fiber} we conclude that the left square in the diagram 
$$
\xymatrix{
P_GX \ar[d] \ar[r] & P \cX \ar[d] \ar[r] & \cX  \ar[d] \\
X  \ar[r] & \cX \ar[r] & \cX \times \cX \times \cX,}
$$
is a fiber product. This implies that the exterior square in  
$$
\xymatrix{
P_GX \ar[d] \ar[r] & X \ar[d] \ar[r] & {*} \ar[d] \\
P \cX  \ar[r] & \cX \ar[r] & [*/G],
}
$$
is also a fiber product. That is, $P\cX \cong [P_GX/G]$, and this concludes the proof. 
\end{proof}

\begin{remark} 
\label{rem:disj}
Assume now that $X$ is affine, $G$ acts linearly and that $\cX = [X/G]$ is a DM stack with finite stabilizers (that is, such that the map $I \cX \rightarrow \cX$ is finite). Under these assumptions we can give a more explicit description of $P \cX$. If $g, h$ are in $G$ let $\Gamma_{g,h}$ be the image of $X$ in $X \times X \times X$ under the assignment: $x \mapsto (x, gx, hx)$. Let $\Delta \subset X \times X \times X$ be the diagonal subscheme.  Then the derived scheme $P_GX$ decomposes as the disjoint union 
$$
P_GX = \underset{g, h \in G}{\coprod} \Gamma_{g, h} \times_{X \times X \times X} \Delta.
$$
\end{remark}

\section{Virtual orbifold K-theory and the proof of the main Theorem}
\label{sec:proof}
\subsection{Virtual orbifold K-theory}
The virtual orbifold cohomology of differential orbifolds was introduced in 
\cite{LUX, LUX+}. Virtual orbifold cohomology is closely related to Chen-Ruan  cohomology, and a precise comparison between the two was obtained in Theorem 1.1 of \cite{LUX+}. In the setting of algebraic geometry, the study of virtual orbifold cohomology  and virtual orbifold K-theory was pursued in \cite{EJK2, EJK3}. 
We start by recalling briefly the setting of \cite{EJK2, EJK3}. 

Let $\cX$ be a smooth Deligne-Mumford stack with finite stabilizers. Assume  that $\cX$ admits a presentation as a global quotient of a smooth affine scheme by a linear algebraic group, $\cX = [X/G]$.  
\begin{definition} \begin{itemize}
\item Denote $I_GX$ the  \emph{inertia scheme} of $\cX$, 
$$
I_GX:= \{(x, g) | gx = x\} \subset X \times G . 
$$
\item Denote $I^2_G X$ the \emph{double inertia scheme} of $\cX$,
$$
I^2_G X:= \{(x, g, h) | gx = hx= x\} \subset X \times G \times G. 
$$
\end{itemize}
\end{definition}

\begin{remark}
\label{rem:disjdisj}
If $g \in G$ denote $X^g$ the \emph{underived} fixed locus of $g$: that is, if $\Gamma_g \subset X \times X$ is  the graph of $g$, set $X^g := \Gamma_g \cap \Delta$. Similarly if $g, h \in G$, set $X^{g,h} := X^g \cap X^h$.  
We can decompose $I_G X$ and $I^2_G X$ as the following disjoint unions: 
$$
I_G X = \underset{g \in G}{\coprod} X^g, \text{  } I_G^2 X = \underset{g, h \in G}{\coprod} X^{g, h}.
$$
\end{remark}

\begin{remark}
Let $L_G X$ and $P_G X$ be as in the statement of Lemma \ref{lem:quot}. Then we have isomorphisms $I_GX \cong t_0(L_G X)$ and $I^2_GX \cong t_0(P_G X)$. In particular, $I^2_GX$ is the \emph{underived} fiber product of $I_GX \rightarrow X \leftarrow I_GX$, and we denote the projections $q_1, q_2 : I^2_G X \rightarrow I_G X$. Further  
$I_G X$ is a $X$-group. We let 
$\mu: I^2_G X \rightarrow X$ be the multiplication and $(-)^{-1} : I_GX \rightarrow I_G X$ be the inverse. Note that in Remark \ref{rem:group} we used these  same notations to denote the projections $I^2 \cX \rightarrow I \cX$, and the multiplication and inverse map of $I \cX$: this should cause no confusion as it will be clear from the context whether we are referring to the inertia variety or to the inertia  stack.   \end{remark}

Note that $I_G X$ and $I^2_G X$ carry a natural action of $G$. This gives   presentations of the inertia and double inertia stack as global quotients: $I \cX \cong [I_G X/G]$ and $I^2 \cX \cong [I^2_G X/G]$. We can describe the sheaf theory of $I \cX$ and $I^2 \cX$ in terms of the equivariant sheaf theory of $I_G X$ and $I^2_G X$. In particular, $K_0(I \cX)$ and $K_0(I^2 \cX)$ are naturally  identified with the equivariant Grothendieck groups $K_G(I_G X)$ and $K_G(I^2_G X)$.

\begin{definition}[\cite{EJK1} Definition 3.1]
Let $\cR$ be a class on  $K_G(I^2_GX)$. Then we define a \emph{product}\footnote{We point out that here \emph{product} stands simply for a binary operation which is in principle neither unital nor associative. The works \cite{EJK1, EJK2} contain a careful study of the conditions on $\cR$ under which 
$*_\cR$ is unital, associative, or has various additional properties.} $*_\cR$ on $K_G(I_G X) \cong K_0(\cX)$ by the assignment: 
$$x, y \in K_G(I_GX), \quad  x *_\cR y = \mu_*(q_1^*x \cdot q_2^*y \cdot \lambda_{-1}(\cR)).$$
\end{definition}

\begin{definition}
[\cite{EJK3}  Definition 2.16, \cite{LUX} Definition 19]
\label{def:virt} 
 Let $u: I^2_G X \subset X \times G \times G \rightarrow X$ be the projection on the first factor. Set 
$
\cB := u^*T_X + T_{I^2_GX} - q_1^*T_{I_GX} - q_2^*T_{I_GX} \in K_G(I^2 \cX), 
$
and $\cR := \lambda_{-1}(\cB^*)$. Then the \emph{virtual orbifold product} $*^{virt}$ is given by $*^{virt} := *_{\cR}$.
\end{definition}

Many properties of the virtual orbifold product have been investigated in \cite{EJK2, EJK3, LUX} and \cite{LUX+}. We list two of the most important here:
\begin{itemize} 
\item The product $*^{virt}$ is unital, associative and commutative (\cite{EJK2} Proposition 4.3.2). 

\item $ K^{virt}(\cX):=K_0(\cX, *^{virt})$ has a structure of $\lambda$-ring (see \cite{EJK3} Corollary 5.17).

\item $ K^{virt}(\cX)$ is a Frobenius algebra (see \cite{LUX+} Theorem 2.3 and \cite{EJK1} Proposition 3.5). 
\end{itemize}

We remark that from the vantage point of Theorem \ref{main:loop}, the first two  properties are a consequence of the fact that 
$K^{virt}(\cX)$ is the Grothendieck group of the $\cE_2$-category $\cC oh(L \cX)$. As for the third point, $\cC oh(L \cX, \otimes^{str})$ is a Frobenius algebra object in the closed symmetric monoidal $\infty$-category of small, stable and split closed 
$\infty$-categories (this is a consequence of \cite{BFN} Proposition 6.3). Thus the same is true of its Grothendieck group, which  is isomorphic to 
$K^{virt}(\cX)$. 

\subsection{The derived double inertia stack and excess intersection}
\label{sec:ddine}
Let $\cX$ be a DM stack satisfying the same assumptions as in the previous section: that is, $\cX$ is smooth, has finite stabilizers and can be presented as the  global quotient $[X/G]$ of an affine scheme by a linear group. 
It will be useful to introduce a derived stack, the \emph{derived double inertia stack}, denoted $\cI^2 \cX$, that in a precise sense interpolates  between 
$P \cX$ and $I^2 \cX$. 
It is possible to describe explicitly $\cO_{P \cX}$ and $\cO_{\cI^2 \cX}$ as classes in 
$K_0(I^2 \cX)$, and we will do this next: the calculation of the class of $\cO_{\cI^2 \cX}$ will be especially important in the proof of Theorem \ref{main:loop}.

\begin{definition}
The \emph{derived double inertia stack} of $\cX$, denoted  $\cI^2 \cX$, is the \emph{derived} fiber product of $I \cX \rightarrow X \leftarrow \cX$, that is $\cI^2 \cX = I \cX \times_\cX I \cX$. 
\end{definition}

\begin{remark}
\label{rem:dine}
We have an equivalence  $t_0(\cI^2 \cX) \cong  I^2 \cX$. The derived double inertia stack $\cI^2 \cX$ can be realized as the quotient of the derived scheme $\cI_G^2 X := I_GX \times_X I_G X$ by the action of $G$: we have
$\cI^2 \cX = [\cI^2_GX /G]$. Note that the geometry of $\cI_G^2 X$ is especially simple as it decomposes as the disjoint union
$
\cI_G^2 X = \underset{g, h \in G}{\coprod} X^g \times_X X^h.
$
\end{remark} 

\begin{lemma}
\label{lem:fiberfiber}
Denote $\iota: I \cX \rightarrow L \cX$ the natural embedding. Let $Y$ be the fiber product of the diagram  
$L \cX \stackrel{ev}{\rightarrow} \cX \stackrel{\iota}{\leftarrow} I \cX$. Denote 
$p_i: P \cX \rightarrow L \cX$, $i=1,2$, the projections as in Remark \ref{rem:group}.  Then the following diagrams are fiber products:  
$$
 \xymatrix{   
Y \ar[d] \ar[r] & P \cX \ar[d]^{p_i}  &&  \cI^2 \cX \ar[d] \ar[r] &  Y \ar[d]  \\
I \cX  \ar[r] & L\cX,  &&  Y  \ar[r] & P \cX. }
$$ 
 \end{lemma}
\begin{proof}
To see that the first diagram is a fiber product consider 
$$
 \xymatrix{
Y \ar[d] \ar[r] & P \cX \ar[d]^{p_1} \ar[r]^{p_2} & L\cX \ar[d]  \\
I \cX  \ar[r] & L \cX \ar[r] & \cX.  }
$$
The right square is a fiber product by Lemma \ref{lem:fiber}, and the exterior one is a fiber product by the definition of $Y$, thus the left square has to be a fiber product as well. 

In order to prove that the second diagram is a fiber product, we write it as the upper-left square in 
$$
\xymatrix{
\cI^2 \cX  \ar[d] \ar[r] & Y \ar[d] \ar[r] & I \cX \ar[d]^ \iota  \\
Y  \ar[r] \ar[d]  & P \cX \ar[r]^{p_2} \ar[d]^{p_1} & L \cX \ar[d] \\
I \cX  \ar[r]^ \iota & L \cX \ar[r] & \cX. }
$$
Note that the top right, bottom right, and bottom left squares are all fiber products by the first part of the Lemma and by Lemma \ref{lem:fiber}. Also the derived double inertia stack  
$\cI^2 \cX$ is the fiber product of $I \cX \rightarrow \cX \leftarrow I \cX$: thus, also the exterior square of the diagram is a fiber product. Standard properties of fiber products then imply that the top left square is also a fiber product, and this concludes the proof.

%We show first that $\cI^2 \cX$ is the fiber product of $Y \rightarrow L \cX \stackrel{\iota}{\leftarrow} I \cX$: that is, the exterior square of the left half of the diagram is a fiber product. For this, restrict to the two squares obtained by removing from the diagram the middle horizontal arrows: the right and exterior squares are both fiber products (since $Y$ is the fiber product of $I \cX \rightarrow \cX \leftarrow L \cX$, and $\cI^2 \cX$ of $I \cX \rightarrow \cX \leftarrow I \cX$), and so the left square is also a fiber product. 
%Now focus on the left half of the diagram. We have just showed that the exterior square is a fiber product, and by the first part of the Lemma the bottom one is as well. Thus the upper-left square is a fiber product, and this concludes the proof. 
\end{proof}

\begin{remark}
\label{rem:commutativity}
The derived stack $\cI^2 \cX$ carries two projections $r_1, r_2: \cI^2 \cX \rightarrow I \cX$, and a multiplication map $r_3: \cI^2 \cX \rightarrow I \cX$. Also we have maps $I^2 \cX \stackrel{j}{\rightarrow} \cI^2 \cX \stackrel{l}\rightarrow P \cX$, where $l$ comes from  the second fiber product of Lemma \ref{lem:fiberfiber}, and $j$ and $i := l \circ j$ are the canonical embeddings of $I^2 \cX = t_0(\cI^2 \cX) = t_0(P \cX)$ into $\cI^2 \cX$ and $P \cX$.  It is important to clarify the relationship between these maps and the various maps to and from $I^2 \cX$ and $L^2 \cX$ that were considered in Remark \ref{rem:group}. We use the notations of Remark \ref{rem:group}: for all $i=1,2, 3$, the following is a commutative diagram 
$$
\xymatrix{
I^2 \cX \ar[r]^j \ar[rd]_{q_i} &  \cI^2 \cX \ar[d]_{r_i} \ar[r]^l & P \cX \ar[d]^{p_i} \\
& I \cX \ar[r] & L \cX.
}
$$
\end{remark}

\begin{lemma}[\cite{CKS} Proposition A.3]
\label{lem:excess}
Let $X, Y, Z$ be schemes and assume that  $Z$ is smooth. Suppose that there are embeddings $X \rightarrow Z \leftarrow Y$, let $W = X \cap Y$, and denote $E$ the \emph{excess (intersection)  bundle}, 
$
E = T_Z|_W / (T_X|_W + T_Y|_W).
$
Then, in $K_0(W)$, we have the identity:   
$$ \sum \limits (-1)^i \pi_i \cO_{X \times_Z Y} = \sum \limits (-1)^ i\Lambda^ i E^\vee =\lambda_{-1}(E^\vee).$$ 
\end{lemma}

\begin{lemma}
\label{prop:excessexcess}
Let $g, h$ be in $G$, and set $W:=X^{g,h}$.  Then:
\begin{itemize}
\item The class in $K_0(W)$ of the excess intersection  bundle $E$ of $X^g$ and $X^h$ in $X$ is given by 
$
E= T X- TX^g -TX^h + TX^{g,h}.  
$ 

\item 
The class in $K_0(W)$ of the excess intersection bundle $E$ of $\Gamma^{g,h}$ and $\Delta$ in $X \times X \times X$ is given by 
$
E =TX + T X^{g,h}. 
$ 

\end{itemize}
\end{lemma}
\begin{proof} 
Both in the statement of the Proposition and in the proof all bundles are always implicitly assumed to be restricted to $W$: $TX^g$, $TX^h$ and $TX^{g,h}$ denote respectively the $g$-invariant, $h$-invariant and $<g,h>$-invariant sub-bundles of $TX|_W$. We start from the first statement: $E$ is by definition the cokernel of the embedding $T X^g + TX^h \rightarrow T X $. Thus the class in $K_0(W)$ of 
$E$ is given by 
$
E = TX - TX^g \oplus TX^h + TX^{g,h} = TX - TX^g - TX^h + TX^{g,h}.
$ 
As for the second statement, the excess bundle is now isomorphic to       the cokernel of the map $$
TX \times TX \rightarrow TX \times TX \times TX, \text{  } (u, v) \mapsto (u + v, gv, hv), 
$$
and therefore $E = 3TX -TX -TX + TX^{g,h} = TX + TX^{g,h}$ in $K_0(W)$.
\end{proof}

\begin{proposition}
\label{prop:excess}
\begin{itemize}
\item 
The excess intersection bundles 
$$
E^{g,h} = TX|_{X^{g,h}} - TX|_{X^{g,h}}^g -TX|_{X^{g,h}}^h + TX|_{X^{g,h}}^{g,h}
$$ 
assemble to a bundle $\cE_{\cI^2}$ on 
$I^2_G X$, and  
$\sum \limits (-1)^ i \pi_i\cO_{\cI^2_GX} = \lambda_{-1}(\cE_{\cI^2}^\vee)$ in $K_0(I^2_GX)$. 
\item The excess intersection bundles 
$$
E^{g,h} = TX|_{X^{g,h}} + TX|_{X^{g,h}}^{g,h}
$$
assemble to a bundle $\cE_P$ on 
$I^2_G X$, and  
$\sum \limits (-1)^ i \pi_i\cO_{P_GX} = \lambda_{-1}(\cE_P^\vee)$ in $K_0(I^2_GX)$.
\end{itemize}
\end{proposition}
\begin{proof} 
By Remark \ref{rem:disj}, Remark \ref{rem:disjdisj} and Remark \ref{rem:dine}, we have decompositions: 
$$
P_G X =  \underset{g, h \in G}{\coprod} \Gamma_{g, h} \times_{X \times X \times X} \Delta, \text{  } \cI^2_G X = \underset{g, h \in G}{\coprod} X^g \times_X X^h \text{ and  } I_G^2 X = \underset{g, h \in G}{\coprod} X^{g, h}.
$$
Thus, by Lemma \ref{lem:excess}, the classes in K-theory of 
$\cO_{\cI^2_GX}$ and 
$\cO_{P_GX}$ can be described in terms of the  excess intersection bundles on each component $X^{g,h}$ of $I^2_GX$. 
These have been calculated in Lemma  
\ref{prop:excessexcess} and coincide with the classes appearing in the claim.  
\end{proof}

The bundles $\cE_{\cI^2}$ and $\cE_P$ carry a canonical $G$-equivariant structure. With slight abuse of notation we keep denoting these bundles $\cE_{\cI^2}$, 
$\cE_P$ also when we 
regard them as objects of the $G$-equivariant category $\cC oh_G(I^2_GX)$, or equivalently of $\cC oh(I^2 \cX)$: in the statement of  
the following Corollary   the notations $\cE_\cI^2$ and $\cE_P$ 
are used in this sense, that is to refer to the corresponding bundles on $I^2 \cX$. 

\begin{corollary}
\label{cor:excess}
\begin{itemize}
\item The class $\sum \limits (-1)^ i \pi_i\cO_{\cI^2 \cX}$ is equal to  
$\lambda_{-1}(\cE_{\cI^2}^\vee)$ in 
$K_0(I^2 \cX)$. 
\item The class $\sum \limits (-1)^ i \pi_i\cO_{P  \cX}$ is equal to $\lambda_{-1}(\cE_{P}^\vee)$ in 
$K_0(I^2 \cX)$.
\end{itemize}
\end{corollary}
\begin{proof}
The claim follows from Proposition \ref{prop:excess}, and the fact that $P \cX$ and $\cI^2 \cX$ are isomorphic to the quotients
$
P \cX = [P_GX/G], \text{  } \cI^2 \cX = [\cI^2_G X /G] $ (see Lemma \ref{lem:quot} and Remark \ref{rem:dine}). 
\end{proof}

\begin{remark}
\label{rem:excess}
Direct inspection reveals that class of $\cE_{\cI^2}$ in $K_0(I^2 \cX)$ coincides with the class $\cB$ that appears in the definition of the virtual orbifold product, see Definition \ref{def:virt}. This observation is a key  ingredient in the proof of Theorem \ref{main:loop}. \end{remark}

\subsection{The proof of the main Theorem }
In this Section we prove that the string tensor product on $\cC oh(\cX)$ categorifies the virtual orbifold product. 
As before, we assume that $\cX$ is a smooth DM stack with finite stabilizers that admits a presentation as the  global quotient $[X/G]$ of an affine scheme by a linear group. It will be important to refer to various maps relating $\cI^2 \cX$, $P \cX$, $I^2 \cX$, $L \cX$ and 
$I \cX$: we let $\iota: I \cX \rightarrow L \cX$ be the natural embedding, and for the rest use the same notations as in Remark \ref{rem:commutativity}.
 
 \begin{lemma}
\label{lem:conv-pro}
Let $\cF$ and $\cG$ be in $\cC oh(\cI \cX)$. 
Then  
$
\iota_*\cF \otimes^{str} \iota_*\cG \cong \iota_*r_{3*} (r_1^*\cF \otimes r_2^*\cG). 
$
\end{lemma}

\begin{proof}
Consider the diagram

$$
\xymatrix{
& & \cI^2 \cX  \ar[rd]^{s_2} \ar[ld]_{s_1} & & \\
& Y \ar[dl]_{n_1} \ar[r]^{u_1} & P \cX \ar[dl]^{p_1} \ar[d]^{p_3} \ar[dr]_{p_2} & Y \ar[l]_{u_2} \ar[dr]^ {n_2} & \\
I \cX  \ar[r]^{\iota} & L\cX & L \cX & L \cX & I\cX  \ar[l]_\iota, }
$$

where the right, left and top squares are all fiber products: note that $\cI^2 \cX$ is the fiber product of the top square by  Lemma \ref{lem:fiberfiber}. As in Remark \ref{rem:commutativity}, we denote $l$ the composition $u_1 \circ s_1 \cong u_2 \circ s_2$.  
 The base change 
formula \cite{To4} Proposition 1.4 gives equivalences $p_1^*\iota_* \cF \cong u_{1*}n_1^*\cF$ and $p_2^*\iota_*\cG \cong u_{2*}n_2^*\cF$. Using Lemma \ref{lem:3} we can rewrite 
$$
p_1^*\iota_* \cF \otimes p_2^*\iota_* \cG \cong u_{1*}n_1^*\cF \otimes u_{2*}n_2^*\cG \cong  u_{1*}s_{1*}(s_1^{*}n_1^*\cF \otimes s_2^{*}n_2^*\cG) 
\cong l_*(r_1^*\cF \otimes r_2^*\cG). 
$$ 
Recall that  $p_3 \circ l \cong \iota \circ r_3$ (see Remark \ref{rem:commutativity}), and thus
$$
\iota_* \cF \otimes^{str} \iota_* \cG  = p_{3*}(p_1^*\iota_*\cF \otimes p_2^* \iota_*\cG ) \cong p_{3*}l_*(r_1^*\cF \otimes r_2^* \cG) \cong \iota_* r_{3*}(r_1^*\cF \otimes r_2^* \cG).
$$
\end{proof}

\begin{lemma}\label{product and derived thickening}
\label{prop:orbi}
Denote 
$- \cdot -$ the product on $K_0(I^2 \cX)$ induced by the ordinary tensor product of sheaves on $I^2 \cX$. Denote $\cD$ the class $\Sigma_i (-1)^ i \pi_i \cO_{\cI^2 \cX}$ in 
$K_0(I^2 \cX)$. If $\cF, \cG$ are in $\cC oh(I \cX)$, then $
q_{3*}(q_1^* \cF \cdot q_2^* \cG \cdot \cD) = r_{3*}(r_1^*\cF \otimes r_2^* \cG) 
$ in $K_0(I \cX)$. 
\end{lemma}

\begin{proof}
As in Remark \ref{rem:commutativity} let $j: I^2 \cX \rightarrow \cI^2 \cX$ be the natural embedding. Recall by Corollary \ref{cor:heart} that the class of $\cO_{\cI ^2 \cX}$ in $G_0(\cI^2 \cX)$ is equal to $j_*\cD$. 
Note also that $r_1^*\cF \otimes r_2^* \cG$ lies in $\cC oh(\cI^2 \cX)$, and that 
$r_{3*}$ descends to a map $r_{3*}: G_0(\cI^2 \cX) \rightarrow K_0(I \cX)$. 
We have the following equalities in $K_0(I \cX)$: 
$$
r_{3*}(r_1^*\cF \otimes r_2^*\cG) = r_{3*}(r_1^*\cF \otimes r_2^*\cG \otimes \cO_{\cI ^2 \cX}) = r_{3*}(r_1^*\cF \otimes r_2^*\cG \otimes j_*\cD) = 
r_{3*} j_*(j^*(r_1^*\cF \otimes r_2^*\cG) \cdot \cD), 
$$
where the last one is a consequence of  the projection formula. Further we can write   
$$
r_{3*} j_*(j^*(r_1^*\cF \otimes r_2^*\cG) \cdot \cD) = r_{3*} j_*((j^*r_1^*\cF \cdot j^*r_2^*\cG) \cdot \cD) = q_{3*}(q_1^*\cF \cdot q_2^* \cG \cdot \cD),
$$
as $q_i = r_i \circ j$ for all $i=1,2,3$ (see Remark \ref{rem:commutativity}) and this concludes the proof. 
\end{proof}

\begin{theorem}
Let $\iota: I \cX \rightarrow L \cX$ be the natural embedding. Then 
$\iota_*$ is an isomorphism of rings:  
$
\iota_*: K^{virt}(\cX) = K_0(I \cX, *^{virt}) \stackrel{\cong}{\rightarrow} G_0(L \cX, \otimes^{str}).
$
\end{theorem}

\begin{proof}
Recall that by Corollary \ref{cor:heart} the map $\iota_*$ is an isomorphism of groups. We need to prove that $\iota_*$ is also compatible with the product structures. 
Let $\cF$ and $\cG$ be in $\cC oh(I \cX)$. By Lemma \ref{lem:conv-pro} and Lemma \ref{prop:orbi}, 
$$
\iota_*\cF \otimes^{str} \iota_*\cG = \iota_*r_{3*}(r_1^*\cF \otimes r_2^* \cG) = \iota_*q_{3*}(q_1^*\cF \cdot q_2^* \cG \cdot \cD),
$$
where $\cD = \Sigma_i (-1)^ i \pi_i \cO_{\cI^2 \cX}$. By Corollary   \ref{cor:excess} there is an identity $\cD=\lambda_{-1}(\cE_{\cI^2}^\vee)$. We pointed out in Remark \ref{rem:excess} that the class of  $\cE_{\cI^2}$ in $K_0(I \cX)$ is equal to the class $\cB$ from Definition \ref{def:virt}: thus, in the notation of Definition \ref{def:virt}, $\cD = \lambda_{-1}(\cE_{\cI^2}^\vee) = \cR$. As a consequence we can rewrite 
$$
q_{3*}(q_1^*\cF \cdot q_2^* \cG \cdot \cD) = q_{3*}(q_1^*\cF \cdot q_2^* \cG \cdot \cR) = \cF *^{virt} \cG.
$$
Applying $\iota_*$, we obtain an  identity $
\iota_*(\cF) \otimes^{str} \iota_*(\cG) = \iota_*(\cF *^{virt} \cG)
$ in $G_0(L \cX)$, 
and this 
 concludes the proof.
\end{proof}

\end{document}